\documentclass[reqno, a4paper]{amsart} 

\usepackage[latin1]{inputenc}
\usepackage{subfigure}

\usepackage{amssymb} 
\usepackage{amsthm}
\usepackage{upref} 
\usepackage{graphicx}
\usepackage{color}
\usepackage{url}
\usepackage{paralist} 
\usepackage{tikz}
\usepackage{listings}

\usepackage[british]{babel}

\usepackage{calc}
 
 \usepackage[all]{xy}

\usepackage{mathrsfs} 

\usepackage[all]{xy}

\usepackage{hyperref}

\hypersetup{pdfauthor={Matthias Lenz},pdftitle={}, 
  pdfkeywords={} }

\newtheorem{Theorem}{Theorem}

\newtheorem{Proposition}[Theorem]{Proposition}
\newtheorem{Lemma}[Theorem]{Lemma}

\theoremstyle{definition}
\newtheorem{Definition}[Theorem]{Definition}

\newtheorem{Example}[Theorem]{Example}

\theoremstyle{remark}
\newtheorem{Remark}[Theorem]{Remark}

\DeclareMathOperator{\DM}{DM}

\newcommand{\abs}[1]{\left|#1\right|}
\DeclareMathOperator{\rank}{rk}

\newcommand{\st}{s.\,t.\ } 
\newcommand{\ie}{\textit{i.\,e.\ }} 
\newcommand{\eg}{\textit{e.\,g.\ }} 

\newcommand{\Z}{\mathbb{Z}}

\newcommand{\R}{\mathbb{R}}

\newcommand{\K}{\mathbb{K}}
\newcommand{\Acal}{\mathcal{A}}

\newcommand{\Dcal}{\mathcal{D}}

\newcommand{\Mcal}{\mathcal{M}}

\newcommand{\Pcal}{\mathcal{P}}

\newcommand{\Gcal}{\mathcal G}

\newcommand{\tutte}{{\mathfrak T}} %
\newcommand{\aritutte}{{\mathfrak M}} %

\newcommand{\latproj}[1]{{\bar #1}} 

\DeclareMathOperator{\hilb}{{{Hilb}}}

\DeclareMathOperator{\LG}{LG}
\DeclareMathOperator{\pr}{pr} 

\DeclareMathOperator{\ind}{Ind} 
\DeclareMathOperator{\indc}{Ind^{cyc}} 
\DeclareMathOperator{\indr}{Ind^{rep}} 
\DeclareMathOperator{\indrr}{\widetilde{Ind}^{rep}} %
\DeclareMathOperator{\fab}{FinAb} %

\title{Stanley--Reisner rings for quasi-arithmetic matroids}
\author{Matthias Lenz}

\thanks{%
The author was supported by a Swiss Government Excellence Scholarship for Foreign
Scholars and subsequently by a fellowship within the postdoc programme of the German
Academic Exchange Service (DAAD)%
}

\date{\today}

\address{Universit\'e de Fribourg, D\'epartement de Math\'ematiques, 1700 Fribourg, Switzerland}
\email{maths@matthiaslenz.eu}

\begin{document}
 
 \begin{abstract}
 In this note we 
 define a Stanley--Reisner ring for quasi-arithmetic matroids and more general structures.
 To this end, we define two types of 
 CW complexes associated with a quasi-arithmetic matroid 
 that generalize independence complexes of matroids.
 Then we use Stanley's construction of Stanley--Reisner rings for simplicial posets.
\end{abstract}

 \maketitle

 \section{Introduction}

 A (quasi-)arithmetic matroid $\Acal$ is a triple $(E, \Delta, m)$, where $(E,\Delta)$ is a matroid on the ground set $E$ with 
 independence complex $\Delta \subseteq 2^E$
 and $m: 2^E\to \Z_{\ge 1}$ is the so-called multiplicity function, that must satisfy certain axioms
 \cite{branden-moci-2014}.
 In the representable case, \ie when the 
 arithmetic matroid is determined by a list of integer vectors,
 this multiplicity function records data such as the 
 absolute value of the determinant of a basis.

 A toric arrangement is an arrangement of subtori of codimension one in a  real or complex torus 
 (\eg \cite{callegaro-delucchi-2017,delucchi-dantonio-2016,deconcini-procesi-2005,ehrenborg-readdy-slone-2009,lawrence-2011,moci-tutte-2012}).
 While matroids capture a lot of combinatorial and topological information about hyperplane arrangements \cite{orlik-terao-1992, stanley-2007},
 arithmetic matroids carry similar information about 
 toric arrangements \cite{callegaro-delucchi-2017,lawrence-2011,moci-tutte-2012}.
 Toric arrangements are particularly important due to their connection with the problem of
 counting lattice points in polytopes.
 This was implicitly discovered in the 1980s 
 by researchers working on splines \cite{BoxSplineBook}
 such as  Dahmen and Micchelli.
 It was recently made more explicit and put in a broader context by De~Concini, Procesi, Vergne, and others 
 \cite{concini-procesi-book}.
 
 Various simplicial complexes can be associated with a matroid in a natural way:
 the matroid complex (or independence complex), the broken circuit complex, and the order complex of the lattice of flats.
 As topological spaces, they are interesting in their own right and they can be used to derive 
 interesting information about the underlying matroid (see \eg \cite{bjoerner-1992}).
 The broken circuit complex has been generalized to arithmetic matroids in the representable case. This construction plays an important role 
 in several papers on the topology of toric arrangements \cite{callegaro-delucchi-2017,delucchi-dantonio-2016,deconcini-procesi-2005}.
 
 The first goal of this paper is to define an independence complex for an arithmetic matroid. The $i$th entry of the 
 $f$-vector of this complex should be equal to the weighted number of independent sets of cardinality $i$ in the underlying matroid, 
 where each independent set is weighted by its multiplicity.
 It is easy to see that this cannot be achieved by a simplicial complex, so
 we will  we construct a CW complex instead.
 In fact, we will provide two constructions: one works for arbitrary quasi-arithmetic matroids (and even more general structures), the other one
 uses the so-called 
 layer groups that can be associated to a representation of an arithmetic matroid \cite{delucchi-riedel-2015,lenz-computing-2017}.
 The $h$-vectors 
 of arithmetic matroids (which are by construction also the $h$-vectors of the  arithmetic independence complexes)
 and related algebraic structures 
 arising in the theory of vector partition functions 
 have already been studied \cite{dall-thesis,lenz-arithmetic-2016}.

 The second goal of this paper is to construct a Stanley--Reisner ring for arithmetic matroids.
 The construction is straightforward: we are able to prove that our arithmetic independence complexes are simplicial posets. Then we employ a construction of Stanley
 that associates a Stanley--Reisner ring to any simplicial poset \cite{stanley-simplicial-1991}.
 In fact, our construction not only works for the two types of arithmetic independence complexes that we define, but for a more general structure that we call
 independence complex defined by a surjective finite abelian group structure on a simplicial complex.
 Just like the Stanley--Reisner ring of a matroid encodes the $h$-vector of the matroid complex, the arithmetic Stanley--Reisner ring
 encodes the $h$-vector of the arithmetic independence complex.

 A large number of algebraic structures that can be associated with  hyperplane arrangements and matroids
 have turned out to be very interesting and useful
 (\eg \cite{ardila-postnikov-2009,berget-2010,brion-vergne-1999,feichtner-yuzvinsky-2004,orlik-solomon-1980,orlik-terao-1994,terao-2002,wagner-1999}).
 There is a strong interest in  
 inequalities for $f$-vectors and $h$-vector of matroid and broken circuit complexes
 (\eg \cite{chari-1997,hibi-1992,kubitzke-le-2016,swartz-2005}) and algebraic tools have been 
 crucial for some of the proofs.
 This includes log-concavity of $f$-vectors and $h$-vectors of matroid complexes and broken circuit complexes
 \cite{adiprasito-huh-katz-2015,ardila-denham-huh-2017,huh-hvector-2015,lenz-mason-2013}.
  Using the fact that the Stanley--Reisner ring of a matroid complex is level, 
  one can deduce certain inequalities for their $h$-vectors
  (\eg \cite[p.~93 and Proposition~3.3(a)]{stanley-1996}).
  Outside of matroid theory, Stanley--Reisner rings %
  also play an important role.
  For example, 
  the equivariant cohomology ring of a complete and simplicial toric variety is isomorphic to the Stanley--Reisner ring
  of the corresponding fan  \cite[12.4.14]{cox-little-schenck-2011}.
  The hard Lefschetz property of certain Stanley--Reisner rings %
  allowed
  Stanley to prove the  $g$-theorem  for rational simplicial polytopes \cite{stanley-gtheorem-1980}.
  This result was conjectured by McMullen in 1971 \cite{mcmullen-1971} and its many generalisations have received a lot of interest 
  \cite{swartz-2014}.

\subsection*{Acknowledgements}
 The author would like to thank Emanuele Delucchi and Ivan Martino for several interesting discussions.

\section{Background}
 In this section we will introduce the mathematical background, \ie 
 Stanley--Reisner rings of simplicial complexes and simplicial posets, matroids, and arithmetic matroids.

 Throughout this paper, $\K$ denotes some fixed field. 
 
\subsection{Stanley--Reisner rings of simplicial complexes}

An (abstract) \emph{simplicial complex} $\Delta$ is a collection of subsets of a finite set $E$.
We denote by $f_i$ the number of elements of $\Delta$ of cardinality $i$. If $\Delta$ is non-empty, there is a maximal integer $r$  \st
$f_r \neq 0$. Then we say that $\Delta$ has \emph{rank} $r$. %
 The \emph{$f$\!-polynomial} $f_\Delta(t)$ and the \emph{$h$-polynomial} $h_\Delta(t)$ of $\Delta$ are defined as 
\begin{align}
  f_\Delta(t) &:= \sum_{i=0}^r f_i t^{r-i} \text{ and}\\
  h_\Delta(t) &:= f_\Delta(t-1) = \sum_{i=0}^r h_i t^{r-i}. 
\end{align}
 The vectors $(f_0,\ldots, f_r), (h_0, \ldots, h_r) \in \Z^{r+1}$ are called the  \emph{$f$-vector} and the \emph{$h$-vector} of
 $\Delta$, respectively.
 The $f$-vector is usually defined in a slightly different way in the topology 
 literature (the index is shifted by $1$), 
 but our notation has some advantages and it is used in some articles on matroid theory (\eg \cite{bjoerner-1992}).

 Let $\Delta$ be a  simplicial complex 
 on the ground set $\{1,\ldots, N\}$.
 The \emph{Stanley--Reisner ideal} $I_\Delta$ 
 and the \emph{Stanley--Reisner ring} (or face ring) 
 $\K[\Delta]$ 
 are defined as 
 \begin{align}
   I_\Delta &:= \bigl( \{s_{i_1}\cdots s_{i_r}: \{i_1,\ldots,i_r\}\not\in \Delta\} \bigr) \subseteq \K[ s_1,\ldots, s_N]  
   \\
  \label{eq:SRR}
  \text{and } \K[\Delta] &:= \K[s_1,\ldots,s_N]/ I_\Delta.
 \end{align}
$\K[\Delta]$ is a graded ring where each variable has degree $1$. It is known that its  Hilbert series is 
\begin{equation}
\label{eq:HSSR}
  \hilb(\K[\Delta],t) = \frac{h_0 + h_1 t + \ldots + h_r t^r}{(1-t)^r}.
\end{equation}
For more information on these topics, see \cite{stanley-1996}.

\subsection{Stanley--Reisner rings of simplicial posets}

Closely following Stanley's paper \cite{stanley-simplicial-1991}, we will now introduce simplicial posets and their Stanley--Reisner rings.
A \emph{simplicial poset}
is a poset $P$ with a unique minimal element $\hat 0$ \st every interval $[\hat 0, y]$ is a boolean algebra
\cite{bjoerner-1984,garsia-stanton-1984,stanley-simplicial-1991}.
All posets considered in this paper are  assumed to be finite.
If a simplicial poset $P$ is in addition a meet-semilattice, then $P$ is just the face lattice (ordered by inclusion) of a simplicial complex.
Simplicial posets are special cases of CW posets, as defined in \cite{bjoerner-1984}.
This implies that a simplicial poset $P$ is the face poset of a regular CW complex $\Gamma$.
We may informally regard $\Gamma$ as a generalized simplicial complex whose faces
are still simplices, but we allow two faces to intersect in any subcomplex of their boundaries, rather than just in a single face.
For a simplicial poset $P$ we can define a grading $\rho$ as follows: for $y\in P$, 
$\rho(y)$ is defined as the rank of the boolean algebra $[\hat 0, y]$.
Let $f_i$ denote the number of $y\in P$ \st $\rho(y)=i$. 
Then we can define $f$ and $h$-vector/polynomial as above.

\begin{Definition}[{\cite[Definition~3.3]{stanley-simplicial-1991}}]
\label{Definition:SimplicialSRR}
Let $P$ be a simplicial poset with elements $\hat 0=y_0,y_1,\ldots, y_p$.
We define the \emph{Stanley--Reisner ideal} of $P$,  $I_P$
to be the ideal 
in the polynomial ring $\K[y_0, y_1,\ldots, y_p]$
generated by the following elements:
\begin{itemize}
 \item[$(S_1)$] $ y_i y_j $ if $y_i$ and $y_j$ have no common upper bound in $P$, and
 \item[$(S_2)$] $ y_i y_j - ( y_i \wedge y_j)(\sum_z z)$, where $z$ ranges over all \emph{minimal} upper bounds of $y_i$ and $y_j$, otherwise;
 \item[$(S_3)$] $ (y_0 - 1) $.
\end{itemize}
 The \emph{Stanley--Reisner ring} of $P$ is defined as $\K[P] =  \K[y_0, y_1,\ldots, y_p] / I_P$.
\end{Definition}
It is clear that $y_i \wedge y_j$ exists whenever $y_i$ and $y_j$ have an upper bound $z$, since the interval $[\hat 0, z]$ is a boolean algebra and therefore a lattice.
It is easy to see that if $P$ is the face lattice of a simplicial complex $\Delta$, then $ \K[ P ] \cong \K[\Delta] $.

Now using the same notation as above, we define a quasi-grading on the ring $ \K[y_0,y_1,\ldots, y_p]$ by setting $\deg(y_i) :=  \rho(y_i)$.
We do not get an actual graded algebra as we have defined it because $\deg(y_0)=0$, so $\dim_\K( \K[y_0, y_1,\ldots, y_p]_0) =2$ instead of $1$. 
The relations $(S_1)$, $(S_2)$ and $(S_3)$ are homogeneous with respect to this grading, 
so $\K[P]$ inherits a grading from $\K[y_0, y_1,\ldots, y_p]$. 
Moreover, since $\deg( y_0 - 1 ) = 0$ it follows that $ \dim_\K( (\K[P])_0 )=1$, so $\K[P]$ is a genuine graded algebra.

\begin{Proposition}[{\cite[Proposition~3.8]{stanley-simplicial-1991}}]
 \label{Proposition:SimplicialSRR}
  Let $P$ be a finite simplicial poset of rank $r$   with $h$-vector 
  $( h_0, h_1, \ldots, h_r)$. With the grading of $\K[P]$ just defined, we have
  \begin{equation}
       \hilb( \K[P], t) = \frac{ h_0 + h_1 t + \ldots + h_r t^r}{  ( 1 - t )^r }.
  \end{equation}

\end{Proposition}

\subsection{Matroids}

 A \emph{matroid complex} (or independence complex) is a simplicial complex $\Delta$ on a finite ground set  $E$
 \st $\emptyset \in \Delta$ and if $A_1, A_2\in \Delta$ and $\abs{A_2} > \abs{A_1}$, then there is 
 $ e  \in A_2 \setminus A_1$ \st $A_1 \cup\{e\}\in \Delta$.
 The pair $(E,\Delta)$ is called a \emph{matroid}.
 Matroids can also be defined in several other equivalent ways
 \cite{MatroidTheory-Oxley}. The \emph{rank function} of a matroid $\Mcal = (E,\Delta)$ is the function 
 $ \rank : 2^E \to \Z_{\ge 0} $ given by $ \rank(A) = \max_{\{S \in \Delta : S \subseteq A\}} \abs{S} $.
 The integer $r:=\rank(E)$ is called the \emph{rank} of $\Mcal$.
 
 Given a matrix $X\in \K^{r\times N}$, one obtains the matroid $\Mcal_X=([N],\Delta_X)$, where $\Delta_X$ denotes the 
 set of all $S\subseteq [N]$ which index a linearly independent set of columns of $X$.
 The matrix $X$ is called a 
 \emph{representation} of  $\Mcal_X$.
 
 The \emph{Tutte polynomial} \cite{brylawski-oxley-1992} of a matroid $\Mcal=(E,\Delta)$ is defined as 
\begin{equation}
  \tutte_\Mcal(x,y) := \sum_{ A \subseteq E } (x-1)^{ r - \rank(A) } (y-1)^{ \abs{A} - \rank(A) } \in \Z[x,y].
\end{equation}
 It is easy to see that $ f_\Delta(t) = \tutte_\Mcal(t+1,1) $, hence $ h_\Delta(t) =  \tutte_\Mcal( t, 1)$.
 Using \eqref{eq:HSSR}, this implies that the Hilbert series of the Stanley--Reisner ring of a matroid complex is encoded 
 by the Tutte polynomial:
 \begin{equation}
  \label{eq:StanleyReisnerTutte}
   \hilb(\K[\Delta],t)=  \frac{ t^{ r } \tutte_\Mcal(\frac 1t,1)}{  (1 - t)^{ r } }.
 \end{equation}

Matroid complexes of been studied intensively and they have many nice properties \cite{bjoerner-1992,stanley-1996}.
In particular, they are shellable \cite{provan-1997}. This implies that their $h$-vectors are non-negative.

\subsection{Arithmetic matroids}
\label{Subsection:AriMatroids}

Arithmetic matroids capture many combinatorial and topological properties of toric arrangements 
 in a similar way as matroids carry information about the corresponding hyperplane arrangements.
\begin{Definition}[%
\cite{branden-moci-2014,moci-adderio-2013}]
 An \emph{arithmetic matroid} is a triple $(E, \Delta, m)$, where
 $(E, \Delta)$ is a matroid and $m : 2^E \to \Z_{\ge 1}$ denotes the
 \emph{multiplicity function}
 that satisfies the following axioms:
 \begin{itemize}
   \item[(P)]  Let $R\subseteq S\subseteq E$. The set $[R,S]:=\{A: R\subseteq A\subseteq S\}$ is called a \emph{molecule} 
 if $S$ can be written as the disjoint union $S=R\cup F_{RS}\cup T_{RS}$ and for each $A\in [R,S]$,
 $\rank(A)= \rank(R) + \abs{A \cap F_{RS}}$ holds.
 
 For each molecule   $[R,S]$, the following inequality holds:
 \begin{equation}
  \label{eq:Paxiom}
  \rho(R,S) := (-1)^{\abs{T_{RS}}} \sum_{A\in [R,S]} (-1)^{\abs{S} - \abs{A} } m(A) \ge 0.
 \end{equation}
  \item[(A1)]  For all $A \subseteq E$ and $e \in E$:
 if $\rank(A \cup \{e\}) = \rank(A)$, then $m(A \cup \{e\}) \big| m(A)$.
 Otherwise, $m(A) \big| m(A \cup \{e\})$.
  \item[(A2)]  If $[R, S]$ is a molecule and $S=R\cup F_{RS}\cup T_{RS}$, then 
  \begin{equation}
     m(R)m(S) = m(R \cup F_{RS} ) m(R \cup T_{RS} ).
  \end{equation}
\end{itemize}
  A \emph{quasi-arithmetic matroid} is a triple $(E, \Delta, m)$, where
 $(E, \Delta)$ is a matroid and the multiplicity function $m : 2^E \to \Z_{\ge 1}$ satisfies  the axioms (A1) and (A2).
\end{Definition}

 The prototypical example of an arithmetic matroid 
 is defined by a list of vectors $X = (x_1,\ldots, x_N) \subseteq\Z^r$, or equivalently, by a matrix $X\in \Z^{r\times N}$.
 In this case, for a subset  $S\subseteq E=[N]:=\{1,\ldots, N\}$ of cardinality $r$ that defines a basis, we have $m(S) = \abs{\det(S)}$
 and in general $ m(S) := \abs{ (\left\langle  S\right\rangle_\R \cap \Z^r) / \left\langle S\right\rangle}$. 
 Here, $\langle S \rangle\subseteq \Z^r$ denotes the subgroup generated  by $\{x_e: e\in S\}$ and $\langle S \rangle_\R\subseteq \R^r$
 denotes the subspace spanned by the same set.

 In fact, one can represent a slightly more general class of arithmetic matroids by a list of elements in a finitely 
 generated abelian group.
 By the fundamental theorem of finitely generated abelian groups, every finitely generated abelian group $G$ is isomorphic
 to $\Z^r \oplus \Z_{q_1} \oplus \ldots \Z_{q_n}$ for  suitable non-negative integers $d, n, q_1,\ldots, q_n$.
 There is no canonical isomorphism $G \cong \Z^r \oplus \Z_{q_1} \oplus \ldots \Z_{q_n}$.
  However, $G$ has a uniquely determined \emph{torsion subgroup} $G_t \cong \Z_{q_1} \oplus \ldots \Z_{q_n}$ consisting of all the torsion elements.
 There is a free group
 $\latproj{G} := G/ G_t \cong \Z^r$.
 For $X\subseteq G$, we will write $\latproj{X}$ to denote the image of $X$ in $\latproj{G}$.

\begin{Definition}
 Let $\Acal = (E, \Delta, m)$ be an arithmetic matroid.
 Let $G$ be a finitely generated abelian group and let $X = (x_e)_{e\in E}$ be a list of elements of $G$.
 For $A\subseteq E$, let  $G_A$ denote the maximal subgroup of $G$ \st $ \abs{G_A / \left\langle A \right\rangle}$ is finite.
 Again, $\left\langle A \right\rangle \subseteq G$ denotes the subgroup generated by $\{x_e:e\in A\}$.

  $ X \subseteq G $ is called a \emph{representation} of $\Acal$ if the matrix $\bar X\subseteq \bar G$ represents the matroid $(E, \Delta)$ 
  and 
  $m_X(A) = m(A) = \abs{G_A / \left\langle A \right\rangle}$ for all $A\subseteq E$.
  The arithmetic matroid $\Acal$ is called \emph{representable} if it has a representation $X$.
\end{Definition}

 Note that if $A$ is an independent set, then $m_X(A) = m_{\bar X}(A) \cdot \abs{G_t}$ (cf.~\cite[Lemma~8.3]{lenz-ppcram-2017}).
An arithmetic matroid $\Acal = (E, \rank, m)$ is called \emph{torsion-free} if 
 $m(\emptyset)=1$. If such an arithmetic matroid is representable, then it can be represented by a list of vectors in lattice,
 \ie a finitely generated abelian group which is torsion-free.

    The \emph{arithmetic Tutte polynomial} \cite{moci-adderio-2013,moci-tutte-2012}
    of an arithmetic matroid $\Acal = (E,\Delta, m)$ 
   is defined as
 \begin{equation}
 \label{equation:aritutte}
  \aritutte_\Acal ( x, y) := \sum_{ A \subseteq E } m(A) ( x - 1 )^{ r - \rank(A)  }  
  ( y - 1 )^{ \abs{A} - \rank(A)  }.
 \end{equation}
 
 If an arithmetic matroid $\Acal$ is represented by a list $X\subseteq \Z^r$ that is totally unimodular, 
 then the multiplicity function is constant and equal to $1$. Hence the arithmetic Tutte polynomial 
 and the Tutte polynomial are equal in this case.

\section{Arithmetic independence complexes}
 
 In this section we will define several simplicial posets that can be used to define the arithmetic Stanley--Reisner ring using Definition~\ref{Definition:SimplicialSRR}.
 In Subsection~\ref{Subsection:Presheafs} we will give a very general definition that includes the other two as special cases. 
 In Subsection~\ref{Subsection:RepresentationComplex} will construct an arithmetic independence complex, given a representation of an arithmetic matroid.
 In Subsection~\ref{Subsection:CyclicComplex} we will describe another construction that works for quasi-arithmetic matroids and even more general structures.
 In general, in the representable case, these two constructions yield different complexes.

\subsection{Simplicial posets defined by a simplicial complex with a surjective finite abelian group structure}
 \label{Subsection:Presheafs}

 Let $\Delta$ be a  simplicial complex on the ground set $E$.
 Let $\fab$ denote the category of finite abelian groups.
 We say that $\Delta$ has a \emph{surjective finite abelian group structure} $\Gcal$
 if there is a functor $\Gcal : \Delta \to \fab$ \st the images of all homomorphisms are
 surjective maps. Here, 
 $\Delta$ is considered as a category in the usual way of turning a poset such as $(\Delta, \subseteq)$ into a category:
 the objects are the elements of $\Delta$ and 
 for $\sigma_1, \sigma_2\in \Delta$, 
 \begin{equation}
    \hom(\sigma_1,\sigma_2) = \begin{cases}  \{(\sigma_1,\sigma_2)\}   & \text{ if } \sigma_1 \subseteq \sigma_2, \\   
       \emptyset & \text{otherwise}.
     \end{cases}
 \end{equation}
 Composition of homomorphisms is defined via  $(\sigma_1, \sigma_2) \circ(\sigma_2, \sigma_3) = (\sigma_1, \sigma_3)$.

\smallskip
 The definition can be rephrased in the following way without using terminology from category theory.
 We say that $\Delta$ has a surjective finite abelian group structure $\Gcal$
 if there is a map
 $\Gcal : \Delta \to \fab$
 \st the following conditions are satisfied:
 \begin{asparaenum}[(i)]
    \item  for each $ S \in \Delta$ and $a\in E\setminus S$ for which $S \cup \{a\} \in \Delta$,
      there is a   map
      $\pi^\Gcal_{ S, a} : \Gcal( S \cup \{a\}) \twoheadrightarrow \Gcal( S)$.
    \item for all $  S\in \Delta$, $a,b\in E \setminus  S$, $a\neq b$ for which 
      $ S\cup \{a,b\}\in \Delta$,
     the following diagram commutes:
 \begin{equation}
 \label{eq:CommutingMaps}
    \xymatrix{
     \Gcal(  S \cup \{ a,b\}) \ar@{->>}[r]^{\pi^\Gcal_{ S \cup \{ a \}, b}} \ar@{->>}[d]^{\pi^\Gcal_{ S \cup\{b\}, a}}  & \ar@{->>}[d]^{\pi^\Gcal_{ S, a}} \Gcal( S \cup \{ a\}) \\
     \Gcal( S \cup \{ b\}) \ar@{->>}[r]^{\pi^\Gcal_{ S, b}} &  \Gcal( S).
  }
 \end{equation}
\end{asparaenum}
 We call $\Gcal$ \emph{torsion-free} if $\abs{\Gcal(\emptyset)}=1$.
 
 \begin{Definition}
   Let $\Delta$ be a simplicial complex on the ground set $E$   
   with a  surjective finite abelian group structure $\Gcal$.
   We define the independence poset $\ind(E, \Delta, \Gcal)$ as follows:
  the ground set is
  \begin{equation}
   \{ ( S, g) :  S \in \Delta, \, g \in \Gcal( S) \}
  \end{equation}
   and for $a \not \in  S$ for which $S\cup \{a\}\in \Delta$, $( S\cup \{a\},g)$ covers $( S, h)$ if and only if $\pi_{ S,a}^\Gcal(g) = h$.
  
 \end{Definition}

  $\ind(E, \Delta, \Gcal)$ is indeed a poset.
  If one defines a poset via cover relations, the only thing that needs to be checked is that the cover
  relations digraph does not contain a cycle, but this is clear from the definition.

 \smallskip
 For $ S = \{s_1,\ldots, s_k\} \in \Delta$, it will be 
 useful to define the map $\pi^\Gcal_ S : \Gcal( S) \to \Gcal(\emptyset)$ by 
 $ \pi^\Gcal_ S (g ) := \pi^\Gcal_{\{s_1,\ldots, s_{k-1}\}, s_k} \Big( \ldots \big( \pi^\Gcal_{\{s_1\}, s_2} 
     (\pi^\Gcal_{\emptyset, s_1} ( g ) )\big)\ldots\Big)$.
  Since the diagrams in \eqref {eq:CommutingMaps} commute, the order of the elements of $S$ does not matter.

 \begin{Lemma}
 \label{Lemma:PosetsSimplicial}
  If $\Gcal$ is torsion-free, then  $\ind(E, \Delta, \Gcal)$ is a simplicial poset with minimal element $(\emptyset, 0)$.
  In general, $\ind(E, \Delta, \Gcal)$ is a disjoint union\footnote{This means that the Hasse diagrams are disjoint unions in the sense of graph theory.} of simplicial posets
  with minimal elements $\{ (\emptyset, g) : g \in \Gcal(\emptyset) \}$.
  
  Furthermore, the $f$\!-vectors of the connected components of $\ind(E, \Delta, \Gcal)$ are pairwise equal.
 \end{Lemma}
 \begin{proof}
     Let $ S \in \Delta$ and $h \in \Gcal(  S )$.
     By definition, there is a unique $g = \pi_S^\Gcal(h) \in \Gcal(\emptyset)$ \st 
     $(\emptyset, g) \le ( S, h)$.
     Hence each connected component of the Hasse diagram contains at least one element of type $(\emptyset,g)$ for $g\in \Gcal(\emptyset)$.
     Now let $g_1,g_2\in \Gcal(\emptyset)$ \st
     $(\emptyset,g_1)$ and $(\emptyset,g_2)$ are in the same connected component.
     This means that there is a path $ p_0 = (\emptyset,g_1), p_1,\ldots, p_l = (\emptyset,g_2)$ in the Hasse diagram of the poset 
     (considered as an undirected graph)
     that connects the two elements,
     \ie for $i=1,\ldots, l $, $p_i \in \ind(E,\Delta, \Gcal)$ and 
     either $p_i$ covers $p_{i-1}$ or $p_{i-1}$ covers $p_{i}$. Since one of two adjacent elements in the path covers the other one, 
     they must lie above the same minimal element. Hence by induction, all elements in the path must lie above the same minimal element.
     Thus $g_1 = g_2$ and we can deduce that the
     Hasse diagram has 
     $\abs{ \Gcal(\emptyset) }$ connected components and each one contains exactly one element of type $(\emptyset, g)$.

     Let $g,h, S$ be as above.
     We still need to show that $[(\emptyset, g), ( S, h)]$ is a boolean algebra. But this is   
     easy: $[ \emptyset,  S ]$ is the standard example of a  boolean algebra. 
     Adding  group elements to each set (a suitable image of $h$) does not change this poset (for this, the fact that the diagrams commute is used).

    Now let $g_1,g_2\in \Gcal(\emptyset)$ and let $ S\in \Delta$. To prove the last statement, it is sufficient to show that
    $\abs{(\pi^\Gcal_{ S})^{-1} ( g_1)} = \abs{ (\pi^\Gcal_{ S})^{-1}( g_2 )}$. 
    But this is clear since $\pi_{ S}^\Gcal$ is a group homomorphism.
 \end{proof}
 Note that the proof of the lemma relies on the fact  that the maps are fixed in advance and everything commutes (`local'
 commutativity as in the case of (non-representable) matroids over a ring \cite{fink-moci-2016} is not sufficient).

 Let $\Gcal$ be a surjective finite abelian group structure on the simplicial complex $\Delta$.
 We will now construct a torsion-free surjective finite abelian group structure $\tilde \Gcal$ on the same simplicial complex.
 Let $ S  \subseteq E$.
 We define $\tilde\Gcal( S)$ as  follows:
 \begin{equation}
    \tilde \Gcal( S) :=  (\pi_ S^\Gcal)^{-1} ( 0 )  =   \{  h \in \Gcal(S) : \pi_S^\Gcal(h)  = 0     \}.
 \end{equation}
 We define the maps
 $\pi^{\tilde\Gcal}_{ S, a} : \tilde\Gcal( S \cup \{a\}) \to \tilde\Gcal( S)$
  by restricting $\pi^{\Gcal}_{ S, a}$.
 By construction, $\pi^{\tilde\Gcal}_{ S, a}$ is surjective, the image is contained in  $\tilde\Gcal( S)$
 and the maps commute as in \eqref{eq:CommutingMaps}. Hence $\tilde\Gcal$ is also a surjective finite abelian group structure on $\Delta$.

  \begin{Theorem}
  \label{Theorem:SRofSheafIndComplex}
    Let $P = \ind(E, \Delta, \Gcal)$ be an independence poset defined by a  surjective finite abelian group structure on a simplicial complex of rank $r$.
    Suppose that $\Gcal$ is  torsion-free, \ie $\Gcal(\emptyset)\cong\{0\}$. Then the 
    Stanley--Reisner ring  $\K[P]$  satisfies
    \begin{equation}
        \hilb( \K[P], t) = \frac{ h_P(t) }{ (1-t)^r }.
    \end{equation}  
    Now suppose that $\Gcal$ has torsion and let $\tilde P = \ind(E,\Delta,\tilde\Gcal)$ denote the independence poset of the corresponding torsion-free 
    surjective finite abelian group structure.
    Then the
    Stanley--Reisner ring  $\K[\tilde P]$ satisfies
 \begin{equation}
      \hilb( \K[\tilde P], t) = \frac{1}{\abs{\Gcal(\emptyset)}} \cdot \frac{ h_{P}(t) }{ (1-t)^r }.
 \end{equation}  
 \end{Theorem}
 \begin{proof}
     This follows directly from 
     Proposition~\ref{Proposition:SimplicialSRR} and Lemma~\ref{Lemma:PosetsSimplicial}.
 \end{proof}

 Below, we will define two types of independence complexes for (quasi-)arithmetic matroids, 
 which will allow us to define the arithmetic Stanley--Reisner ring whose Hilbert series encodes the 
 arithmetic $h$-vector.

 \begin{Theorem}
  \label{Theorem:SRofAriIndComplex}
    Let $P$ be a poset for which one of the two following statements holds:
    \begin{itemize}
     \item let $X\in \Z^{r \times N}$ be a matrix that represents 
      an arithmetic matroid  $\Acal  = (E_X,\Delta_X, m_X)$ of rank $r$ and 
        $P = \indr(E_X, \Delta_X, \Gcal_X)$ is the independence poset derived from $X$ (cf.~Subsection~\ref{Subsection:RepresentationComplex}). 
     \item let $\Acal = (E,\Delta, m )$ be a (weak quasi-)arithmetic matroid of rank $r$ and 
          $P = \indc(E, \Delta, \Gcal)$  is its cyclic independence poset (cf.~Subsection~\ref{Subsection:CyclicComplex}).
    \end{itemize}
    In both cases, the following statements hold for the arithmetic Stanley--Reisner ring $\K[P]$:
 
    suppose $\Acal$ is torsion-free, \ie $m(\emptyset)=1$. Then 
  \begin{align}
    \hilb( \K[P], t) &= \frac{ \aritutte_\Acal(t,1) }{ (1-t)^r }. \\
%
 \text{In general, we have }
      \hilb( \K[\tilde P], t) &= \frac{1}{m(\emptyset)} \cdot \frac{ \aritutte_{\Acal}(t,1) }{ (1-t)^r }, 
 \end{align}  
 where $\tilde P$ denotes the independence poset of the torsion-free structure associated with $\Acal$.
 \end{Theorem}

\begin{proof}%
 This follows directly from Theorem~\ref{Theorem:SRofSheafIndComplex}
 and Lemma~\ref{Lemma:RepPresheaf} and Lemma~\ref{Lemma:CycPresheaf}  below.
\end{proof}

\subsection{Arithmetic independence complexes derived from a representation}
 \label{Subsection:RepresentationComplex}
 
  Let $X \in \Z^{r\times N}$ be a matrix. 
  For $S\subseteq [N]$, let
  $X[S]$ denote the submatrix of $X$ whose columns are indexed by $S$.
   Following  \cite{delucchi-riedel-2015,lenz-computing-2017}, we define
  \begin{align}
    W(S) &:=  X[S]^T \R^r \cap \Z^r,
    \displaybreak[2]\\
    I(S) &:= X[S]^T \Z^r, \displaybreak[2]\\
    Z(S) &:= \Z^S / I(S), \text{ and} \displaybreak[2]\\
    \LG(S) &:= W(S) / I(S). 
  \end{align}
  The groups $\LG(S)$ for $S\subseteq [N]$ are called \emph{layer groups}, as they encode a lot of information about the 
  poset of layers of the corresponding toric arrangement.
 $\LG(S)$ is the torsion subgroup of $Z(S)$  \cite{delucchi-riedel-2015}.
 For $s\in [N] \setminus S$, let
 $\pr_{S,s} : \Z^{S\cup \{s\}} \to \Z^S$ denote the projection that forgets the coordinate corresponding to $s$.
 This induces a map $\bar\pr_{S,s} : \LG(S \cup \{s\}) \to \LG(S)$
 (\cite[Lemma~2 and Lemma~3]{lenz-computing-2017}).

\begin{Remark}
\label{Remark:LayerGroupLatticePoints}
  Let $X \in \Z^{r\times N}$ be a matrix. 
  It follows from the definition that if $S\subseteq [N]$ is independent, then 
  the set of lattice points in the half-open parallelepiped $\Pcal_S$ that is spanned by the columns of $X[S]$ is a set of representatives
  for the elements of $\LG(S)$.
  If $S\cup \{s\}$ is also independent, then $ \Pcal_S $ is a facet of $ \Pcal_{ S\cup \{s\} }$.
  Note that the projection map $\pr_{S, s}$ sends 
  a point in $\Pcal_{S\cup \{s\}}$ to a point in $\Z^S$ that is not necessarily contained in $\Pcal_S$,
  but of course it has a representative modulo $I(S)$ that is contained in $\Pcal_S$
  (cf.~Figure~\ref{Figure:TA}).  
\end{Remark}

\begin{Lemma}
  \label{Lemma:RepPresheaf}
  Let $X \in \Z^{r \times N}$ be a matrix. 
  Let $E = [N]$ and let $\Delta_X$ be the matroid complex  of the matroid represented by $X$.
  
  For $S\in \Delta_X$, let $\Gcal_X(S):= \LG(S)$.
  Furthermore, for $s\in E\setminus S$ for which $S\cup \{s\}\in \Delta_X$, 
  let $\pi_{ S, s}^\Gcal := \bar\pr_{ S, s} : \Gcal(S\cup\{s\}) \to \Gcal(S)$.

    Then $(E,\Delta_X, \Gcal_X)$ is a surjective finite abelian group structure on $\Delta$.

  \end{Lemma}
 \begin{proof}
   By assumption, $S \cup \{s\}$ is independent.  
       Hence by  \cite[Lemma~5]{lenz-computing-2017}, $ \pi^\Gcal_{S,s} : \Gcal_X(S \cup \{s\}) \to \Gcal_X(S)$ is surjective. 
       To see that the projection maps commute in the required way, note that this is trivial
       for the projection $\pr_{S,s} : W(S\cup\{s\}) \to W(S)$ and we can easily pass over to the quotient since
       $\pr_{S,s} (I(S \cup \{s\})) = I(S )$ \cite[Lemma~2]{lenz-computing-2017}.
\end{proof}
 
 \begin{Definition}[Arithmetic independence complex derived from a representation]
   Let $X \in \Z^{r\times N}$ be a matrix and let  $E$, $\Delta_X$, and $\Gcal_X$ be as in Lemma~\ref{Lemma:RepPresheaf}.
   Let $\Acal_X$ be the arithmetic matroid represented by $X$.
           Then we call 
            $\indr(X):=(E_X,\Delta_X,\Gcal_X)$ the \emph{arithmetic independence complex} of $\Acal_X$ derived from the representation $X$.

   Now let $X$ be a list of elements of a finitely generated abelian group $G$ of rank $r$.
   Let $\bar X$ denote the projection of $X$ to a lattice as explained in Subsection~\ref{Subsection:AriMatroids}.
   Then we call ${\indrr}(\bar X):=(E,\Delta_{\bar X},\Gcal_{\bar X})$ the\footnote{Note 
   that while the construction of $\Gcal_{\bar X}$ requires the choice of a non-canonical isomorphism $G / G_t \cong \Z^r$,
   where $G/G_t$ denotes the lattice containing the elements of $\bar X$, the torsion-free independence complex is independent of this choice.
   }
   \emph{torsion-free arithmetic independence complex} of  $\Acal_X$ derived from the representation $X$.  
 \end{Definition}

 \begin{Example}
 \label{Example:RepIndComplex}
  Let us consider the matrix $X=\begin{pmatrix} 2 & -2 \\ 2 & 2 \end{pmatrix}$. 
 \begin{figure}[t]
  \begin{center}
   \input{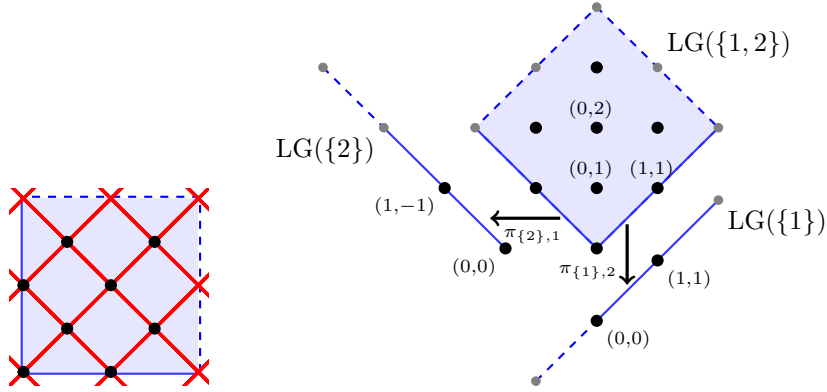} \,
   \newlength{\vertexmain}
\setlength{\vertexmain}{2.6pt}
\newlength{\vertexsec}
\setlength{\vertexsec}{2pt}
\begin{tikzpicture}[line join=round, scale=0.8,                                                                                                                                                                                                                         
              vertex/.style={color=black},                                                                                                                                                                                                                            
              secondary_vertex/.style={color=gray},                                                                                                                                                                                                                   
              line/.style={very thick},                                                                                                                                                                                                                               
              cube/.style={color=blue!80!white, thick},                                                                                                                                                                                                               
              cube_thin/.style={color=blue,  thick, dashed}                                                                                                                                                                                                           
              ]                                                                                                                                                                                                                                                       

   \draw[draw=none,fill = blue!10!white, ultra thick] (0,0) -- (2,2) -- (0,4) -- (-2,2) -- cycle; 
   \draw[cube_thin] (0,0) -- (2,2) -- (0,4) -- (-2,2) -- cycle;                                                                                                                                                                                                   
   \draw[cube] (-2,2) -- (0,0) -- (2,2);                                                                                                                                                                                                   
   \filldraw [vertex] (0, 0) circle (\vertexmain);
   \filldraw [vertex] (-1, 1) circle (\vertexmain);
   \filldraw [vertex] (0, 1) node[anchor=south, black] {$\scriptstyle (0,1)\;\;$} circle (\vertexmain);
   \filldraw [vertex] (1, 1) node[anchor=south, black] {$\scriptstyle (1,1)\;\;$} circle (\vertexmain);
   \filldraw [vertex] (-1, 2) circle (\vertexmain);
   \filldraw [vertex] (0, 2) node[anchor=south, black] {$\scriptstyle (0,2)\;\;$} circle (\vertexmain);
   \filldraw [vertex] (1, 2) circle (\vertexmain);
   \filldraw [vertex] (0, 3) 
      circle  (\vertexmain);
   \filldraw [secondary_vertex] (-2, 2) circle (\vertexsec);
   \filldraw [secondary_vertex] (-1, 3) circle (\vertexsec);
   \filldraw [secondary_vertex] ( 0, 4) circle (\vertexsec);
   \filldraw [secondary_vertex] ( 1, 3) node[anchor=south west, black] {$\LG(\{1,2\})$} circle (\vertexsec);
   \filldraw [secondary_vertex] ( 2, 2) circle (\vertexsec);

   \draw[cube] ( 0 -1.5, 0 ) -- (-2 - 1.5, 2);                                                                                                                                                                                                   
   \draw[cube_thin] ( -3 -1.5, 3 ) -- (-2 - 1.5, 2);                                                                                                                                                                                                   
   \filldraw [vertex] ( 0 -1.5, 0 ) node[anchor=north east, black] {$\scriptstyle (0,0)$}  circle (\vertexmain);
   \filldraw [vertex] (-1 - 1.5, 1) node[anchor=north east, black] {$\scriptstyle (1,-1)$}  
      circle (\vertexmain); 
   \filldraw [secondary_vertex] (-2 - 1.5, 2) node[black, anchor=north east] {$\LG(\{2\})$}  circle (\vertexsec);
   \filldraw [secondary_vertex] (-3 - 1.5, 3)  
      circle (\vertexsec);
   \draw[very thick, ->] (-0.6, 0.5) --  node[anchor=north] {$\;\;\scriptstyle\pi_{\{2\},1}$} (-1.75, 0.5);                                                                                                                                                                                                   

   \draw[cube_thin] (-1, -1 - 1.2) -- (0, 0 - 1.2);                                                                                                                                                                                                   
   \draw[cube] (0,0 - 1.2) -- (2,2 - 1.2);                                                                                                                                                                                                   
   \filldraw [vertex] (0, 0-1.2) node[anchor=north west, black] {$\scriptstyle (0,0)$} circle (\vertexmain);
   \filldraw [vertex] (1, 1 - 1.2) node[anchor=north west, black] {$\scriptstyle (1,1)$} circle (\vertexmain); 
   \filldraw [secondary_vertex] ( 2, 2 - 1.2) node[black, anchor=north west] {$\LG(\{1\})$} circle (\vertexsec);
   \filldraw [secondary_vertex] ( -1, -1 - 1.2) circle (\vertexsec);
   \draw[very thick, ->] ( 0.5, 0.4) --  node[anchor=north east] {$\scriptstyle \pi_{\{1\},2}$} (0.5, -0.6); 
\end{tikzpicture}
   \end{center}
 \caption{Left: the toric arrangement  defined in Example~\ref{Example:RepIndComplex} drawn in $\R^2/\Z^2$.
   Right: the corresponding layer groups, represented by lattice points in half-open parallelepipeds.}
 \label{Figure:TA}
 \end{figure}
We denote the columns by $a$ and $b$.
    The arithmetic Tutte polynomial is $\aritutte_X(x,y)= x^2 + 2 x + 5$.
  For the face poset $P$ of the representable independence complex this implies
    $f_P(t) = \aritutte_X(t+1,1) = t^2 + 4t + 8$ and 
    $h_P(t) = \aritutte_X(t,1) = t^2 + 2 t + 5$.
     The $f$-vector $(f_0,f_1,f_2)=(1,4,8)$ can also be read 
     off from the     corresponding toric arrangement, which is shown in Figure~\ref{Figure:TA}:
     there are $f_2=8$ simple vertices (\ie vertices
     contained in exactly $r=2$ $1$-dimensional subtori), there are $f_1=4$ connected $1$-dimensional subtori, 
     and the torus has $f_0=1$ connected component.
   
 Now let us construct $\indr(X)$. Representatives of the group elements of 
 $\LG(\{1,2\})$ are given by the lattice points of the half-open parallelepiped spanned by the columns of $X$, \ie the set
 $\{(0,0),(-1,1),(0,1),(1,1),(-1,2),(0,2),(1,2),(0,3)\}$. Furthermore, $\LG(\{1\})\cong \LG(\{2\}) \cong \Z/2\Z = \{\bar 0,\bar1\}$.
 Now we calculate the images of $\LG(\{1,2\})$ in the other two groups, \ie we forget a coordinate and reduce the other one modulo $2$:
 \\[2mm]
 \begin{tabular}{l|cccccccc}
  $\LG(\{ 1,2\})$ &  (0,0) & (-1,1) & (0,1) & (1,1) & (-1,2) & (0,2) & (1,2) & (0,3) \\\hline
  $\LG(\{1\})$ &  0 & 1 & 0 & 1 & 1 & 0 & 1 & 0 \\ 
  $\LG(\{2\})$ &  0 & 1 & 1 & 1 & 0 & 0 & 0 & 1
 \end{tabular}\\[2.5mm]
This allows us to construct the poset, which is shown in Figure~\ref{Figure:PosetRep}.

The Stanley--Reisner ideal is
   \begin{equation}
   \begin{split}
    I_P &= 
    \bigl( a_0 C_{-1,1}, a_0C_{11}, a_0C_{1,2}, a_0 C_{-1,2}, a_1 C_0,\ldots, \\
    & \qquad 
    a_0a_1, b_0b_1, \quad C_0 C_{-1,1},\ldots, 
    C_{0,2} C_{-1,2},
    \\ & \qquad
        a_0 b_0 - ( C_0 + C_{0,2}), a_0 b_1 - (  C_{0,1} + C_{0,3}  ), 
    \\ & \qquad \quad
        a_1b_0 - ( C_{-1,2} + C_{1,2}  ), a_1 b_1 - ( C_{1,1} + C_{-1,1} )
   \bigr).
   \end{split}
  \end{equation}
  
 \end{Example}

\begin{figure}
\begin{center}
        \begin{tikzpicture}[line join=round, scale=0.6]
            \node (0) at (0, -0.5) {$\hat{0}$};
            \node (a0) at (-3,1) {$a_0$};
            \node (a1) at (-2,1) {$a_1$};
            \node (b0) at (2,1) {$b_0$};
            \node (b1) at (3,1) {$b_1$};
            \node (ab00) at (-7,4) {$C_{0}$};
            \node (ab-11) at (-5,4) {$C_{-1,1}$};
            \node (ab01) at (-3,4) {$C_{0,1}$};
            \node (ab11) at (-1,4) {$C_{1,1}$};
            \node (ab-12) at (7,4) {$C_{-1,2}$};
            \node (ab02) at (5,4) {$C_{0,2}$};
            \node (ab12) at (3,4) {$C_{1,2}$};
            \node (ab03) at (1,4) {$C_{0,3}$};
            \draw (0) -- (a0);
            \draw (0) -- (a1);
            \draw (0) -- (b0);
            \draw (0) -- (b1);
            \draw (a0) -- (ab00);
            \draw (a0) -- (ab01);
            \draw (a0) -- (ab02);
            \draw (a0) -- (ab03);
            \draw (b0) -- (ab00);
            \draw (b0) -- (ab-12);
            \draw (b0) -- (ab02);
            \draw (b0) -- (ab12);
            \draw (a1) -- (ab-11);
            \draw (a1) -- (ab11);
            \draw (a1) -- (ab-12);
            \draw (a1) -- (ab12);
            \draw (b1) -- (ab-11);
            \draw (b1) -- (ab01);
            \draw (b1) -- (ab11);
            \draw (b1) -- (ab03);
        \end{tikzpicture}
\end{center}
        \caption{$\indr(X)$  as defined in Example~\ref{Example:RepIndComplex}.}
   \label{Figure:PosetRep}
\end{figure}
\begin{figure}
\begin{center}
        \begin{tikzpicture}[line join=round, scale=0.6]
            \node (0) at (0, 0) {$\hat{0}$};
            \node (a0) at (-3,1) {$a_0$};
            \node (b0) at (-2,1) {$b_0$};
            \node (a1) at (2,1) {$a_1$};
            \node (b1) at (3,1) {$b_1$};
            \node (ab0) at (-6,3) {$C_0$};
            \node (ab2) at (-4.5,3) {$C_2$};
            \node (ab4) at (-3,3) {$C_4$};
            \node (ab6) at (-1.5,3) {$C_6$};
            \node (ab7) at (6,3) {$C_7$};
            \node (ab5) at (4.5,3) {$C_5$};
            \node (ab3) at (3,3) {$C_3$};
            \node (ab1) at (1.5,3) {$C_1$};
            \draw (0) -- (a0);
            \draw (0) -- (a1);
            \draw (0) -- (b0);
            \draw (0) -- (b1);
            \draw (a0) -- (ab0);
            \draw (a0) -- (ab2);
            \draw (a0) -- (ab4);
            \draw (a0) -- (ab6);
            \draw (b0) -- (ab0);
            \draw (b0) -- (ab2);
            \draw (b0) -- (ab4);
            \draw (b0) -- (ab6);
            \draw (a1) -- (ab1);
            \draw (a1) -- (ab3);
            \draw (a1) -- (ab5);
            \draw (a1) -- (ab7);
            \draw (b1) -- (ab1);
            \draw (b1) -- (ab3);
            \draw (b1) -- (ab5);
            \draw (b1) -- (ab7);
        \end{tikzpicture}
\end{center}
   \caption{$\indc(X)$  as defined in Example~\ref{Example:CycIndComplex}.}
   \label{Figure:PosetCyclic}
\end{figure}
 
\subsection{Cyclic arithmetic independence complexes}
 \label{Subsection:CyclicComplex}

 We call $\Acal = (E,\Delta, m)$ a
 \emph{weak quasi-arithmetic matroid}
 if $(E,\Delta)$ is a matroid and $m :  \Delta \to \Z_{\ge 1}$ satisfies 
 $ m(  S ) | m(  S \cup \{x\} )$ for  all $S\subseteq E$ and $x\in E$ for which $ S \cup \{x\}\in \Delta$, \ie $m$ is required to satisfy only 
 one of the axioms of a quasi-arithmetic matroid.
 We require only this axiom to define the cyclic arithmetic  independence complex.
 The arithmetic Tutte polynomial of a weak quasi-arithmetic matroid is defined as in \eqref{equation:aritutte}.
 
 \begin{Definition}[Cyclic arithmetic independence complex]
   Let $\Acal=(E,\Delta, m)$ be a weak quasi-arithmetic matroid.
   For $S\in \Delta$, let $\Gcal_m( S) := \Z / m( S)\Z$ and for $a\in E\setminus S$ for which $S\cup \{a\}\in \Delta$,
   let $\pi^{\Gcal_m}_{ S,a}$ be the canonical projection map that sends $\bar 1$ to $\bar 1$.     
   
   Then we define the \emph{cyclic arithmetic independence complex} of $\Acal$ as $\indc(E,\Delta,m):= \ind(E, \Delta, \Gcal_m)$.
 \end{Definition}
 \begin{Lemma}
  \label{Lemma:CycPresheaf}
    Let  $\Acal = (E,\Delta, m)$ be a weak quasi-arithmetic matroid. 
  Then  $\Gcal_m( S)$ is indeed a surjective finite abelian group structure on $\Delta$.
 \end{Lemma}
 \begin{proof}
  Let  $S\subseteq E$, $a\in E\setminus S$, and $ S \cup \{a\}\in \Delta$.
  Since $ m(  S ) | m(  S \cup \{x\} )$, the map $\pi_{S,a}^{\Gcal_m} : \Z / m(S\cup \{a\}) \Z \to \Z / m(S)$ that sends  $\bar 1$ to $\bar 1$
  is indeed a surjective group homomorphism. It is easy to see that the commuting squares condition is satisfied.
 \end{proof}

 \begin{Example}
   \label{Example:CycIndComplex}
  Let us consider again the arithmetic matroid $\Acal$ that was defined in Example~\ref{Example:RepIndComplex}.
  Its cyclic arithmetic independence complex $\indc(\Acal)$ is shown in 
  Figure~\ref{Figure:PosetCyclic}.
    The commuting square is:
  \begin{equation}
    \xymatrix{
     & \Z/8\Z \ar@{->>}[rd]^{} \ar@{->>}[ld]^{}  &    \\
   \Z/2\Z \ar@{->>}[rd]^{} &   & \ar@{->>}[ld]^{} \Z/2\Z. \\ 
   &  \{0\} &
  }
 \end{equation}
  
  The Stanley--Reisner ideal is
  \begin{equation}
   \begin{split}
    I_P &= \bigl( a_0C_1, a_0C_3, a_0C_5, a_0C_7,\: b_0C_1, b_0C_3, b_0C_5, b_0C_7, \: a_1 C_0,\ldots
    \\ &\qquad    
    a_0a_1, a_0b_1, b_0a_1, b_0b_1, \quad C_0C_1, C_0C_2, \ldots, C_6C_7, \\ &\qquad
            a_0b_0 -(C_0 + C_2 + C_4 + C_6), a_1b_1 - (C_1 + C_3 + C_5 + C_7)
   \bigr).
   \end{split}
  \end{equation}

 \end{Example}

 \begin{Example} 
 \label{Example:nonP}
    Let us consider a quasi-arithmetic matroid whose $h$-vector has negative entries.
    Let   $ \Acal = ( \{1,2\}, 2^{\{1,2\}}, m)$, with   $m(\emptyset)=1$ and $m(\{1\})=m(\{2\})=m(\{1,2\}) = 2$. 
    The underlying matroid is the uniform matroid $U_{2,2}$.
  The arithmetic Tutte polynomial is  $\aritutte_\Acal(x,y) = x^2 + 2 x  - 1 $.

Let us construct $\indc(X)$. The commuting square of cyclic groups and the poset are shown in Figure~\ref{Figure:Poset-NonP}.
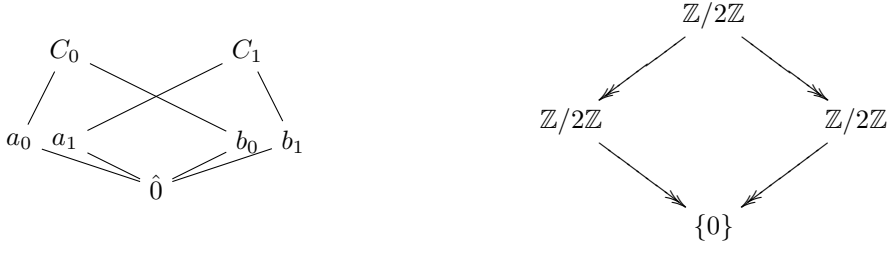
\begin{figure}
\begin{center}
\begin{minipage}{0.49\textwidth}
        \begin{tikzpicture}[line join=round, scale=0.6]
            \node (0) at (0, 0) {$\hat{0}$};
            \node (a0) at (-3,1) {$a_0$};
            \node (a1) at (-2,1) {$a_1$};
            \node (b0) at (2,1) {$b_0$};
            \node (b1) at (3,1) {$b_1$};
            \node (C0) at (-2,3) {$C_0$};
            \node (C1) at (2,3) {$C_1$};
            \draw (0) -- (a0);
            \draw (0) -- (a1);
            \draw (0) -- (b0);
            \draw (0) -- (b1);
            \draw (a0) -- (C0);
            \draw (b0) -- (C0);
            \draw (a1) -- (C1);
            \draw (b1) -- (C1);
        \end{tikzpicture}
\end{minipage}
\begin{minipage}{0.49\textwidth}
          \begin{equation*}
    \xymatrix{
     & \Z/2\Z \ar@{->>}[rd]^{} \ar@{->>}[ld]^{}  &    \\
   \Z/2\Z \ar@{->>}[rd]^{} &   & \ar@{->>}[ld]^{} \Z/2\Z \\ 
   &  \{0\}  &
  }  
 \end{equation*}
\end{minipage}
\end{center}
  \caption{The cyclic arithmetic independence complex  
  and the commuting square of cyclic groups
  corresponding to Example~\ref{Example:nonP}.}
   \label{Figure:Poset-NonP}
\end{figure}
  One can read off the $f$-vector from the arithmetic Tutte polynomial or from the poset.
  It is $(1,4,2)$ and the $h$-vector is $(1,2,-1)$.
  
  For the Stanley--Reisner ideal and the Stanley--Reisner ring we obtain
  \begin{align}
   I_P &=  \bigl( a_0 b_1, a_1 b_0,  a_0a_1, b_0b_1, \, C_0 C_1, a_0 C_1, b_0 C_1, a_1 C_0, b_1 C_0,\,
                                 a_0b_0 - C_0, a_1b_1 - C_1  \bigr) \notag \\
  \K[P] &\cong  \K[ a_0, a_1, b_0, b_1] / ( a_0 b_1, a_1 b_0, a_0a_1, b_0b_1 ).             
  \end{align}
 As a vector space, $\K[P]$ is isomorphic to 
  $ \K[ a_0,b_0  ] \cup \K[ a_1,b_1 ]$.
 Thus, its Hilbert series is 
\begin{align}
   1 + 2 \sum_{i\ge 1}  (i+1) t^i &= 1 + 4t + 6t^2 + 8 t^3 + 10t^4 +\ldots 
   \\ &
   = ( 1 + 2t + 3t^2 + 5t^3 + \ldots)(-t^2 + 2t + 1)  = \frac{ \aritutte_\Acal(t,1) }{(1-t)^2}.
\end{align}
 \end{Example}

 \begin{Remark}
  In general, given a representable arithmetic matroid, the two arithmetic independence complexes that we have defined 
  are not isomorphic.
 For example, $\hat 0$ is a cut vertex of the Hasse diagram of the face poset of the cyclic independence complex in  
 Example~\ref{Example:CycIndComplex}. This is not the case for the representable independence complex in 
 Example~\ref{Example:RepIndComplex}, which was constructed using the same arithmetic matroid.

  The representable independence complex appears to be more interesting as it preserves more structure, \eg the 
  matroid over $\Z$ structure of the matrix (cf.~\cite{fink-moci-2016}). 
 \end{Remark}

 \begin{Remark}
  It is possible to compute the poset of layers of a toric arrangement as follows:
  first construct the arithmetic independence complex derived from a representation and then identify independent sets that 
  correspond to the same flat
  \cite{lenz-computing-2017}.
  It is an open problem whether there is also 
  a 'poset of layers of an arithmetic matroid', in analogy with the lattice of flats of a matroid that exists, even if there is no representation.
  While we are able to construct 
  an arithmetic independence complex in the non-representable case, the surjective finite abelian group structure
  on a simplicial complex defined using the cyclic groups can in general not 
  be extended in a way required for the construction of a poset of layers.
  One would need to assign a cyclic group $\Gcal( S)$ of cardinality $m(S)$ to \emph{each} $ S \subseteq E$ \st
  for $a\in E\setminus  S$, there is a surjection
  $\Gcal( S \cup \{a\}) \twoheadrightarrow \Gcal( S )$ if $ \rank( S \cup \{a\} ) = \rank( S )+1 $,  and
  an injection $\Gcal( S \cup \{a\}) \hookrightarrow \Gcal( S )$ otherwise. In addition, the usual commuting squares conditions would need to be satisfied.
  In the next paragraph we will see that in general, this is not possible.
  
  Let us consider the arithmetic matroid
  $\Acal=(\{1,2\}, \{\emptyset, \{1\}, \{2\}\}, m)$,
  where the 
  multiplicity function $m$ is given by $m(\emptyset)=m(\{1,2\})=3$, $m(\{1\})=6$, and $m(\{2\})=9$.
  This arithmetic matroid can be represented by the list $X=((2,\bar 0), (3, \bar 0))\subseteq \Z\oplus \Z/3\Z$.
  The arithmetic Tutte polynomial is $\aritutte_\Acal(x,y) = 3x + 3y + 9$.   
  The cyclic groups yield the following square:
   \begin{equation}
    \xymatrix{
     &   \Z/3\Z \ar@{^{(}->}[rd]^{\iota_2} \ar@{_{(}->}[ld]_{\iota_1}  &    \\
   \Z/6\Z \ar@{->>}[rd]_{\pi_1} &   & \ar@{->>}[ld]^{\pi_2}  \Z / 9 \Z. \\ 
   &  \Z / 3\Z &
  }
 \end{equation}
 Since we require $\iota_2$ to be injective, 
  $\iota_2( \bar 1) \in \{ \bar 3, \bar 6\}$ must hold. Hence $\pi_2(\iota_2(\bar 1))=\bar 0$.
  On the other hand, $\iota_1(\bar 1) \in \{ \bar 2, \bar 4\}$, which implies $ \pi_1(\iota_1(\bar 1)) \in \{\bar 1, \bar 2\}$. Hence the diagram does not commute.
  A commuting square for this arithmetic matroid can be obtained by replacing the group $\Z/9\Z$ by $\Z / 3 \Z \oplus \Z/ 3\Z$.
\end{Remark}

\section{Future directions and related work}

\begin{Remark}
 The construction of arithmetic independence posets and hence, of Stanley--Reisner rings can be extended
 to the setting of group actions on semimatroids, that where recently introduced by 
 Delucchi and Riedel \cite{delucchi-riedel-2015}.
 This will be done in an upcoming paper of D'Al\`i and Delucchi.
\end{Remark}
 
 \begin{Remark}
  As stated in the introduction, in the case of matroids, algebraic structures associated to matroids have been 
  used to prove inequalities for their $f$ and $h$-vectors.
  For arithmetic matroid, only very little is currently known about the shape of these vectors  
  \cite{wang-yeh-zhou-2015}.
  It would be interesting to prove stronger inequalities.
 However, one cannot expect  results such as log-concavity:
 let $\alpha\in \Z_{\ge 1}$ and let 
 $X= (e_1,e_2, \ldots, e_{r-1}, e_1+\ldots + e_{r-1} + \alpha e_d)\subseteq \Z^r$.
 The arithmetic $f$-polynomial is  
 $f_X(t) = t^r + \binom{r}{1} t^{r-1} + \ldots + \binom{r}{r-1} t + \alpha$
    \cite{lenz-thesis}.
 For $ r \ge 4 $ and sufficiently large $\alpha$, the coefficients of this polynomial are not unimodal and hence also not log-concave. 
 
 Proving that our arithmetic independence posets are Cohen--Macaulay would imply  
 that their $h$-vector is positive
 \cite[Theorem~3.10]{stanley-simplicial-1991}.
 In the case of arithmetic matroids, this follows from the positivity of the coefficients 
 of the arithmetic Tutte polynomial, which is known \cite{backman-lenz-2016,moci-adderio-2013}.
\end{Remark}

  \begin{Remark}
  The Stanley--Reisner ring of the matroid represented by a matrix $X$ is isomorphic 
  to the equivariant cohomology ring of a hypertoric varieted associated with $X$
  \cite[Theorem~3.2.2]{proudfoot-2008}.
  It would be nice to find a similar interpretation of the arithmetic Stanley--Reisner ring
  defined by an integer matrix $X$.
  A somewhat similar goal has been achieved in 
  \cite{deconcini-procesi-vergne-2013}, where it was shown that  
  equivariant $K$-theory and cohomology of certain spaces is related
  to algebraic structures that use Dahmen--Micchelli spaces as building blocks.
  This was done both in the continuous and the discrete case, which correspond to the matroid and the arithmetic matroid case.
  \end{Remark}

  \begin{Remark}
  Let $X \in \R^{d\times N}$ be a matrix and let $\Delta^*$ be the independence complex of the dual matroid.
  The quotient of $\K[\Delta^*]$  by a linear system of parameters is isomorphic to the continuous Dahmen--Micchelli space $\Dcal(X)$
  \cite{procesi-concini-2008}. 
  This suggests a similar relationship between the 
  Stanley--Reisner ring of an arithmetic matroid that is represented by an integer matrix $X \in \Z^{d \times N}$
  and the discrete Dahmen--Micchelli space $\DM(X)$ (see  \cite{concini-procesi-book} for a definition), 
  or its dual, the  periodic $\Pcal$-space \cite{lenz-arithmetic-2016}.
  These two spaces already have the correct Hilbert series, \ie $t^{N-d} \aritutte_{\Acal_X}(1,\frac{1}{t})$.
  The discrete Dahmen--Michelli space $\DM(X)$ can be decomposed into `local' $\Dcal(X)$-spaces that are attached to the vertices of the toric 
  arrangement \cite[(16.1)]{concini-procesi-book}. Is there also a decomposition of the arithmetic Stanley--Reisner ring (derived from  a representation)
  into `local' matroid Stanley--Reisner rings attached to the vertices of the toric arrangement?
 \end{Remark}

\begin{Remark}
 A recent draft of Martino \cite{martino-2017} also suggest a method to construct 
 an arithmetic Stanley--Reisner ring,
 given a representation of an arithmetic matroid.
 The approach is somewhat similar to ours, but instead of layer groups, the abelian groups ($\Z$-modules) that Fink and Moci assign to a list of
 integer vectors in the context of matroids over $\Z$ \cite{fink-moci-2016} are used.
 This structure is in some sense dual to layer groups (\cite[Theorem~D]{delucchi-riedel-2015}),
 so one could expect that the arithmetic independence complex obtained in this way is isomorphic to our 
 arithmetic independence complex derived from a representation. However, the author was not able to verify this for the following reason: 
 the definition in \cite{martino-2017} relies on the fact 
 that every finitely generated abelian group is isomorphic to  $\Z^{d} \times H$ for some
 $d\in \Z_{\ge 0}$ and a finite group $H$. But this isomorphism is not canonical and 
 it is not specified in \cite{martino-2017}, hence the poset does not appear to be well-defined.
\end{Remark}

\renewcommand{\MR}[1]{} 
\bibliographystyle{amsplain}
\providecommand{\bysame}{\leavevmode\hbox to3em{\hrulefill}\thinspace}
\providecommand{\MR}{\relax\ifhmode\unskip\space\fi MR }
\providecommand{\MRhref}[2]{%
  \href{http://www.ams.org/mathscinet-getitem?mr=#1}{#2}
}
\providecommand{\href}[2]{#2}

\end{document}